\documentclass[11pt]{article}
\usepackage{latexsym}
\usepackage{amsfonts}
\usepackage{amsmath}
\usepackage{amssymb}
\usepackage{multirow}
\usepackage{color}
\usepackage{graphicx}
\usepackage{subfigure}
\usepackage{xspace}
\usepackage{fullpage}
\usepackage{microtype}
\usepackage{float}

\newcommand\blackslug{\hbox{\hskip 1pt \vrule width 4pt height 8pt depth 1.5pt
        \hskip 1pt}}
\newcommand\bbox{\hfill \quad \blackslug \medbreak}

\newenvironment{proof}{\noindent {\bf Proof:\ }}{{\quad \blackslug \medbreak}}

\newtheorem{theorem}{}[section]

\title{Clique numbers of graph unions}
\author{Maria Chudnovsky\thanks{Partially supported by NSF grants DMS-1001091 and IIS-1117631.} \, and Juba Ziani\\
Columbia University, New York NY 10027}

\begin{document}
\maketitle
\begin{abstract}
Let $B$ and $R$ be two simple graphs with vertex set $V$, and let $G(B,R)$ be 
the simple graph with vertex set $V$, in which two vertices are adjacent if 
they are adjacent in at least one of $B$ and $R$. For $X \subseteq V$, we 
denote by 
$B|X$ the subgraph of $B$ induced by $X$; let $R|X$ and $G(B,R)|X$ be defined 
similarly. We say that the pair $(B,R)$ is {\em additive} if for every 
$X \subseteq V$, the  sum of the clique numbers of $B|X$ and $R|X$ is at least 
the clique number of $G(B,R)|X$. In this paper we give a necessary and 
sufficient characterization of additive pairs of graphs. This is a numerical 
variant of a structural question studied in \cite{ABC}.

\end{abstract}
\section{Introduction}

All graphs in this paper are finite and simple.  A {\em clique} in a graph $G$ 
is a set of pairwise adjacent vertices; and $\omega(G)$ denotes the largest 
size of a clique in $G$. The {\em complement} $G^c$ of $G$ is the graph with
vertex set $V(G)$, so that two vertices are adjacent in $G^c$  if and only if 
they are non-adjacent in $G$. A {\em stable set} of $G$ is a clique of $G^c$. 
For a subset $X$ of $V(G)$, the graph $G|X$ is a subgraph of $G$ induced by 
$X$. For a graph $H$, we say that $G$  {\em contains $H$} if some induced 
subgraph of $G$ is isomorphic to $H$. If $G$ does not contain $H$, then $G$ is 
{\em $H$-free}. If $\mathcal{H}$ is a family of graphs, then
$G$ is {\em $\mathcal{H}$-free} if $G$ is $H$-free for every 
$H \in \mathcal{H}$.

Let $B$ and $R$ be graphs with vertex set $V$. We denote by  $G(B,R)$ the 
graph with vertex set  $V$, in which two vertices are adjacent if  they are 
adjacent in at least one of $B$ and $R$. In \cite{ABC}
the following question is studied: which graphs $B$ and $R$ have the property 
that every clique of $G(B,R)$ can be expressed as the union of a clique of $B$ 
and a clique of $R$? The main result there is (here $C_k$ denotes a cycle on 
$k$ vertices):

\begin{theorem}
\label{union}
Let $B$ and $R$ be two graphs with vertex set $V$, and suppose that some clique
$G(B,R)$ cannot be expressed as the union of a clique of $B$ and a 
clique  of $R$. Then either 
\begin{itemize}
\item one of $B, R$ contains $C_4$, or 
\item both $B$ and $R$ contain $C_5$.
\end{itemize}
\end{theorem}

We remark that both outcomes of \ref{union} are necessary, because of the 
following two constructions. First, let $B|X$ be 
isomorphic to $C_4$  for some $X \subseteq V$,  and $R|X=B^c|X$; then 
$X$ is a clique in $G(B,R)$, and yet $X$ cannot be expressed as the union of a 
clique of $B$ and a clique of $R$. Similarly, let $B|X$ be 
isomorphic to $C_5$  for some $X \subseteq V$,  and $R|X=B^c|X$ (and thus 
$R|X$ is also isomorphic to $C_5$); then again $X$ is a clique in $G(B,R)$, and 
yet $X$ cannot be expressed as the union of a clique of $B$ and a clique of $R$.

Our goal here is to address a variant of this question,
where we are only interested in the sizes of the cliques. We say that
the pair $(B,R)$ is {\em additive} if for every $X \subseteq V$,

$$\omega(B|X)+\omega (R|X) \geq \omega(G(B,R)|X).$$

The following is immediate:
\begin{theorem}
\label{clique}
Let $B$ and $R$ be two graphs with vertex set $V$. The pair $(B,R)$ is 
additive if and only if for every clique $X$ of $G(B,R)$
$$\omega(B|X)+\omega (R|X) \geq |X|.$$
\end{theorem}

Please note that if $B|X$ is 
isomorphic to $C_4$  for some $X \subseteq V$,  and $R|X=B^c|X$, 
then $\omega(B|X)=\omega(R|X)=2$, and thus $$\omega(B|X)+\omega (R|X)=|X|.$$
Thus our goal here is to refine the first outcome of \ref{union}, in 
order to obtain a characterization of additive pairs.

Let us start by describing a few graphs that we need. For a graph
$G$ and two disjoint subsets $X$ and $Y$ of $V(G)$, we say that $X$
is {\em $G$-complete ($G$-anticomplete)} to $Y$ if every vertex of $X$ is
adjacent (non-adjacent) to every vertex of $Y$. If $|X|=1$, say $X=\{x\}$, we 
write ``$x$ is $G$-complete ($G$-anticomplete) to $Y$'' instead of ``$\{x\}$ is
$G$-complete ($G$-anticomplete) to $Y$''. When there is no risk of confusion,
we write ``complete'' (``anticomplete'')  instead of ``$G$-complete''
(``$G$-anticomplete'').

Let $\mathcal{F}$ be the family of graphs with vertex set
$\{a_1,a_2,a_3,b_1,b_2,b_3\}$ where $\{a_1,a_2,a_3\}$ and $\{b_1,b_2,b_3\}$ are 
cliques, $a_i$ is non-adjacent to $b_i$ for $i \in  \{1,2,3\}$, and the 
remaining adjacencies are arbitrary.

Let $P_0$ be the graphs with vertex set
$\{a_1,a_3,a_3,b_1,b_2,b_3,c\}$ where 
\begin{itemize}
\item $\{a_1,a_2,a_3\}$ is a clique, 
\item $\{b_1,b_2,b_3\}$ is a stable set, 
\item for $i \in \{1,2,3\}$, $b_i$ is non-adjacent to $a_i$, and complete
to $\{a_1,a_2,a_3\} \setminus \{a_i\}$,
\item $c$ is adjacent to $b_1$, and has no other neighbors in $P_0$.
\end{itemize}
Let $P_1$ be the graph obtained from $P_0$ by adding the edge $cb_2$, and let
$P_2$ be the graph obtained from $P_1$ by adding the edge $cb_3$. Let
$\mathcal{P}=\{P_0,P_1,P_2\}$.

\newpage 

\begin{figure}[!htb]
\minipage{0.32\textwidth}
  \includegraphics[width=\linewidth]{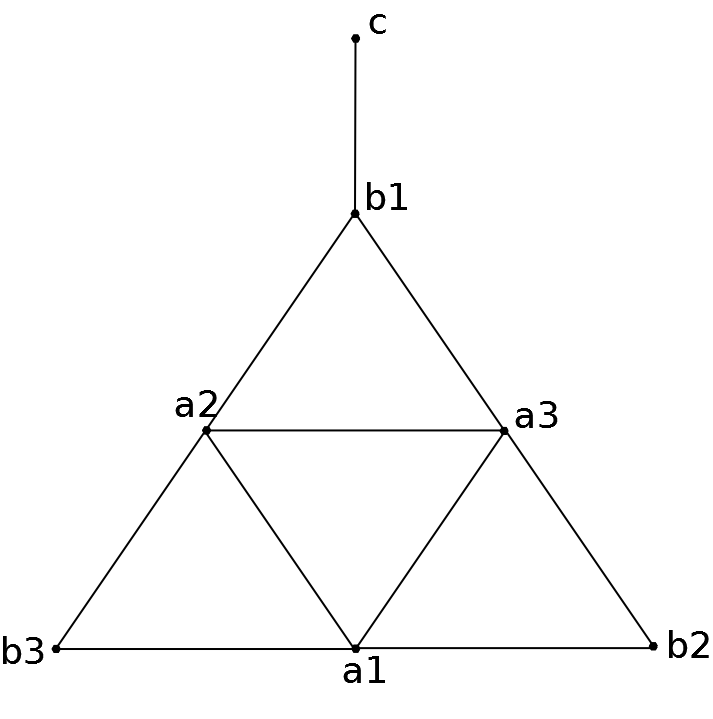}
  \caption{$P_0$}\label{fig:P0}
\endminipage\hfill
\minipage{0.32\textwidth}
  \includegraphics[width=\linewidth]{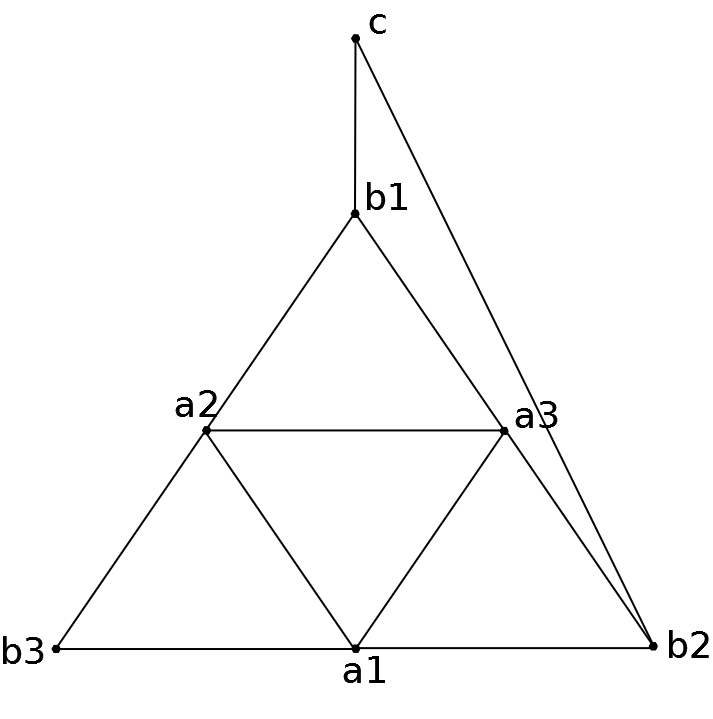}
  \caption{$P_1$}\label{fig:P1}
\endminipage\hfill
\minipage{0.32\textwidth}%
  \includegraphics[width=\linewidth]{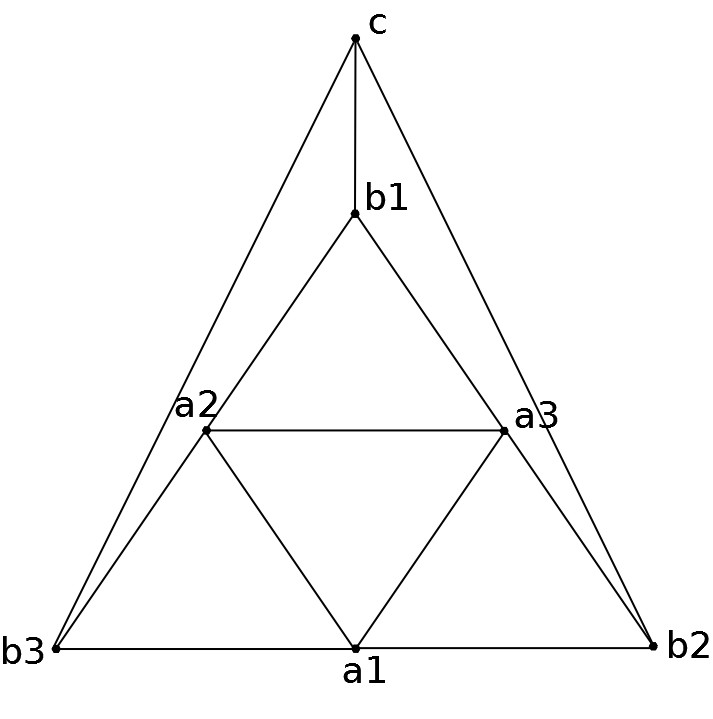}
  \caption{$P_2$}\label{fig:P2}
\endminipage
\end{figure}

\vspace*{5mm}

We can now state our main result.

\begin{theorem}
\label{main}
Let $B$ and $R$ be two graphs with vertex set $V$. Then either
the pair $(B,R)$ is additive, or 
\begin{enumerate}
\item one of $B,R$ contains a member of $\mathcal{F}$, or
\item both $B$ and $R$ contain $C_5$, or
\item both $B$ and $R$ contain $P_0^c$, or
\item $B$ contains $P_0^c$, and $R$ contains a member of $\mathcal{P}$, or
\item $R$ contains $P_0^c$, and $B$ contains a member of $\mathcal{P}$.
\end{enumerate}
\end{theorem}

Let us show that, similarly to \ref{union}, all the  outcomes of 
\ref{main} are necessary. Taking $B$ to be a member of $\mathcal{F}$ (or
$C_5$), and taking $R=B^c$, we construct a pair that is not additive,
and that satisfies only \ref{main}.1(or only \ref{main}.2).
Next, let $B=P_0^c$, and let $R$ be the graph obtained from $=B^c$ by adding
the edge $ca_1$; then
$(B,R)$ is not additive, and it only satisfies \ref{main}.3. Finally, let 
$B=P_0^c$, and let $R$ be the graph obtained from $B^c$ by adding
none, one or both of the edges $cb_2$ and $cb_3$; then the pair $(B,R)$ is not
additive, and it only satisfies \ref{main}.4. Clearly, \ref{main}.5
is just \ref{main}.4 with the roles of $R$ and $B$  reversed.

\section {Proof of \ref{main}}

In this section we prove~\ref{main}. Write $\omega_R=\omega(R)$ and 
$\omega_B=\omega(B)$. Suppose \ref{main} is false, and let 
$B$ and $R$ be two graphs with vertex set $V$ be such that the pair 
$(B,R)$ is not additive, and 
\begin{itemize}
\item both $B,R$ are $\mathcal{F}$-free, and
\item at least one of $B$ and $R$ is $C_5$-free, and
\item at least one of $B$ and $R$ is $P_0^c$-free, and
\item $B$ is $P_0^c$-free or  $R$ is $\mathcal{P}$-free, and
\item $R$ is $P_0^c$-free, or $B$ is $\mathcal{P}$-free, and
\item $B$ and $R$ are chosen with $|V|$ minimum subject to the
conditions above.
\end{itemize}
Write $|V|=n$. By~\ref{clique}, the minimality of $|V|$ implies that $G(B,R)$ 
is a complete graph with vertex set $V$, and $\omega_R+\omega_B < n $. 
Consequently, neither of $B,R$ is a complete graph, and so, 
since every pair of vertices of $V$ is adjacent in $G(B,R)$, we deduce that
$\omega_R \geq 2$, and $\omega_B \geq 2$.

\begin{theorem}
\label{n>5}
$n \geq 6$
\end{theorem}

\begin{proof}
Suppose $n \leq 5$.
Since both $\omega_R \geq 2$, and $\omega_B \geq 2$, and 
$\omega_R+\omega_B < n$, it follows that $|V|=5$, and 
$\omega_R=\omega_B=2$. But then both $B$ and $R$ are isomorphic to $C_5$,
a contradiction. This proves~\ref{n>5}.
\end{proof}

Let $$K=max_{v \in V} {\omega(B \setminus v)},$$
and $$L=max_{v \in V} {\omega(R \setminus v)}.$$

\begin{theorem}
\label{KL}
$\omega_B=K$ and  $\omega_R=n-K-1$. Moreover, for every $v \in V$, 
$\omega(B \setminus v)=K$ and $\omega (R \setminus v)=n-K-1$.
Similarly, $\omega_R=L$, $\omega_B=n-L-1$, and  for every $v \in V$, 
$\omega(R \setminus v)=L$ and $\omega (B \setminus v)=n-L-1$.
\end{theorem}

\begin{proof}
Since the second statement of \ref{KL} follows from the first by reversing the
roles of $B$ and $R$, it is enough to prove the first statement. Let $v \in V$.
Since 
$$n>\omega_B+\omega_R \geq K +\omega(R \setminus v),$$
it follows that  $\omega(R \setminus v) \leq n-K-1$.
On the other hand, it follows from the minimality of $|V|$, that
$$n-1 \leq \omega(B \setminus v)+ \omega(R \setminus v) 
\leq K+\omega(R \setminus v),$$
and so $\omega(R \setminus v) \geq n-K-1$.
Thus $\omega (R \setminus v)=n-K-1$, and $\omega(B \setminus v)=K$.
Finally, since $\omega_B \geq K$, and 
$\omega_R \geq \omega(R \setminus v)=n-K-1$, and $n>\omega_B+\omega_R$,
it follows that $\omega_B = K$, and  $\omega_R =n-K-1$.
This proves~\ref{KL}.
\end{proof}

\ref{KL} immediately implies the following:

\begin{theorem}
\label{K>1}
$K \geq 2$ and $L \geq 2$.
\end{theorem}

\begin{proof}
\ref{K>1} follows immediately from~\ref{KL} and the remark preceding \ref{n>5}. 
\end{proof}

We will need two new graphs:
let $B \setminus R$ be the graph with vertex set $V$, such that two vertices 
are adjacent in $B \setminus R$ if and only if they are adjacent in $B$ and
non-adjacent in $R$. Similarly, let $R \setminus B$ be the graph with vertex 
set $V$, such that two vertices are  adjacent in $R \setminus B$ if and only if 
they are adjacent in $R$ and non-adjacent in $B$.

For a graph $G$ and two disjoint subsets $X$ and $Y$ of $V(G)$ with $|X|=|Y|$, 
we say  that $X$ is {\em matched} to $Y$ if there is a matching 
$e_1, \ldots, e_{|X|}$ of $G$, so that for all $i \in \{1, \ldots, |X|\}$,
the edge $e_i$ has one end in $X$ and the other in $Y$.

\begin{theorem}
\label{matching}
Let $K_1,K_2$ be cliques of size $K$ in $B$. Then $K_1 \setminus K_2$ and 
$K_2 \setminus K_1$ are matched in $R \setminus B$.
\end{theorem}

\begin{proof}
Suppose not. Let $k=|K_1 \setminus K_2|=|K_2 \setminus K_1|$. Then by Hall's 
Theorem \cite{Hall}, there exists $Y \subseteq K_1 \setminus K_2$ and 
$Z \subset K_2 \setminus K_1$ such that $|Z|>k-|Y|$, and $Y$ is
$R\setminus B$-anticomplete to $Z$. Since $G(B,R)$ is a complete graph, it 
follows that $Y$ is $B$-complete to $Z$. But then 
$(K_1 \cap K_2)  \cup Y \cup Z$ is a clique of size at least $K+1$ in
$B$, contrary to~\ref{KL}. This proves~\ref{matching}.
\end{proof} 

\ref{matching} implies the following:

\begin{theorem}
\label{dif<3}
Let $K_1,K_2$ be cliques of size $K$ in $B$. Then 
$|K_1 \setminus K_2| \leq 2$.
\end{theorem}

\begin{proof}
Suppose $|K_1 \setminus K_2| \geq 3$, and let 
$a_1,a_2,a_3 \in K_1 \setminus K_2$ be all distinct.
By~\ref{matching}, there exist $b_1,b_2,b_3 \in K_2 \setminus K_1$, all 
distinct, such that
the sets $\{a_1,a_2,a_3\}$ and $\{b_1,b_2,b_3\}$ are matched in $R \setminus B$.
But then $B|\{a_1,a_2,a_3,b_1,b_2,b_2\}$ is isomorphic to a member of 
$\mathcal{F}$,  a contradiction. This proves~\ref{dif<3}.
\end{proof}

In view of~\ref{KL}, for every $v \in V$, let $K_v$ be a clique of
size $K$ in $B \setminus v$. 

\begin{theorem}
\label{dif2}
There exist $u,w \in V$ such that $|K_u \setminus K_w|=2$.
\end{theorem}

\begin{proof}
Let $v \in V$. By~\ref{K>1}, $K \geq 2$, and so
there exist distinct vertices $u,w \in K_v$. By~\ref{dif<3}, we may assume that
$|K_v \setminus K_u|=1$, where $K_v \setminus K_u=\{u\}$. Let
$x$ be the unique vertex of $K_u \setminus K_v$. Similarly, we may assume that
$|K_v \setminus K_w|=1$, and $K_v \setminus K_w=\{w\}$. 
Let $y$ be the unique vertex of $K_w \setminus K_v$. By~\ref{matching}
$ux$ is an edge $R \setminus B$,  and so $u$ is non-adjacent
to $x$ in $B$. Since $y,u \in K_w$, it follows that $u$ is adjacent to $y$ in 
$B$;  consequently $x \neq y$, and so $x \not \in K_w$. But now both $x$ and 
$w$ are in  $K_u \setminus K_w$, and \ref{dif2} holds. 
\end{proof}

In view of~\ref{dif2}, let $u,w \in V$ be such that 
$|K_u \setminus K_w|=|K_w \setminus K_u|=2$. Write 
$K_u \cap K_w = \{v_3, \ldots, v_K\}$, and $K_i=K_{v_i}$.
In the next theorem we study the structure of the cliques $K_i$.

\begin{theorem}
\label{Ki}
Assume $K \geq 3$. Then there exist vertices $x_1,x_2 \in K_u \setminus K_w$,
$y_1,y_2 \in K_w \setminus K_u$, and 
$p_3, \ldots, p_K \in V \setminus (K_u \cup K_w)$ such that
\begin{enumerate}
\item for every $i \in \{3, \ldots, K\}$ 
$$K_i=((K_u \cap K_w) \cup \{p_i,x_1,y_1\}) \setminus \{v_i\}.$$
\item $\{x_2,y_2,p_3, \ldots, p_K\}$ is a clique of size $K$ in $R \setminus B$.
\item Write $Y=\{x_1,y_1,v_3, \ldots,v_K\}$ and  $Z=\{x_2,y_2,p_3 \ldots, p_K\}$.
Then the pairs $x_1y_2,x_2y_1$ and $v_ip_i$ for $i \in \{3,\ldots, K\}$
are adjacent in $R \setminus B$, and all other pairs $zy$
with $z \in Z$ and $y \in Y$ are adjacent in  $B$.
\end{enumerate}
\end{theorem}

\begin{proof}
Let $K_u \setminus K_w = \{x_1,x_2\}$, $K_w \setminus K_u = \{y_1,y_2\}$.
Fix $i \in \{3, \ldots, K\}$. Then $v_i \in K_u \setminus K_i$, and so
by~\ref{matching}, there exists $p_i \in K_i \setminus K_u$ such that
$v_ip_i$ is an edge of $R \setminus B$. Consequently, 
$p_i \not \in K_u \cup K_w$. Also by~\ref{matching}, the sets $\{x_1,x_2\}$
and $\{y_1,y_2\}$ are matched in $R \setminus B$. Since
$B|\{x_1,x_2,y_1,y_2,v_i,p_i\}$ is not isomorphic to a member of $\mathcal{F}$,
it follows that $p_i$ is not $B$-complete to either $\{x_1,x_2\}$  or
$\{y_1,y_2\}$. From the symmetry we may assume that $p_ix_2$ and $p_iy_2$
are both edges of $R \setminus B$. Therefore, $x_2,y_2 \not \in K_i$, and
so, by~\ref{dif<3}, $K_u \setminus K_i = \{x_2,v_i\}$ and
$K_w \setminus K_i=\{y_2,v_i\}$. Consequently.
$$K_i=((K_u \cap K_w) \cup \{p_i,x_1,y_1\}) \setminus \{v_i\},$$
as required.

Next, since $p_iv_i$ is an edge of $R \setminus B$, and $p_i$ is
$B$-complete to $(K_u \cap K_v) \setminus \{v_i\}$, it follows that
the vertices $p_3, \ldots, p_K$ are all distinct. 

Now let $j \in \{3, \ldots, K\} \setminus \{i\}$.
By the argument in the first paragraph of the proof applied to 
$j$ instead of $i$, we deduce that there exist $k,m \in \{1,2\}$ such that
$$K_j=((K_u \cap K_w) \cup \{p_j,x_k,y_m\}) \setminus \{v_j\},$$
To prove \ref{Ki}.1, it remains to show that $k=m=1$. Suppose not.
Since $x_1,y_1 \in K_i$ it follows that $x_1y_1$ is an edge  of $B$. On the 
other hand, \ref{matching} implies that $x_1y_2$ and $x_2y_1$ are edges of 
$R \setminus B$. Since $K_j$ is a clique of $B$, we deduce that
$x_ky_m$ is an edge of $B$, and so $k=m=2$. But then 
$K_i \setminus K_j = \{p_i,x_1,y_1\}$, contrary to~\ref{dif<3}.
This proves that $k=m=1$, and thus proves~\ref{Ki}.1.

Next, to prove \ref{Ki}.2  suppose that $\{x_2,y_2,p_3, \ldots, p_K\}$ is not
a clique of $R \setminus B$.  We showed earlier that 
$\{x_2,y_2\}$ is $R\setminus B$-complete to $\{p_3, \ldots, p_K\}$, and
that $p_3, \ldots, p_K$ are all distinct.
Suppose first that there exist $k,m \in \{3, \ldots, K\}$ such that
$p_kp_m$ is not an edge of $R\setminus B$. Then
$$X=((K_u \cap K_w) \cup \{p_k,p_m,x_1,y_1\}) \setminus \{v_k,v_m\}$$
is a clique of size $K$ in $B$, but $X \setminus K_u=\{p_k,p_m,y_1\}$,
contrary to~\ref{dif<3}. This proves that $\{p_3, \ldots, p_K\}$
is a clique of $R \setminus B$. Since $\{p_3, \ldots, p_K\}$ is 
$R \setminus B$-complete to $\{x_2,y_2\}$, but $\{x_2, y_2, p_3, \ldots, p_K\}$
is not a clique of $R \setminus B$, 
it follows that $x_2y_2$ is not an edge of 
$R \setminus B$, and therefore $x_2$ is adjacent to $y_2$ in $B$. Consequently,
$Z=(K_u \cup \{y_2\}) \setminus \{x_1\}$ is a clique of size $K$ in $B$.
But now $K_3 \setminus Z=\{x_1,y_1,p_3\}$, contrary to~\ref{dif<3}. This
proves~\ref{Ki}.2.

We now  prove the final statement of~\ref{Ki}. We have already shown that
$x_1y_2,x_2y_1$ and $v_ip_i$ for $i \in \{3,\ldots, K\}$
are adjacent in $R \setminus B$. Next we  observe that every other pair $(z,y)$ 
with $z \in Z$ and $y \in Y$ is contained in  at least one of the cliques
$K_u,K_v,K_3, \ldots, K_K$, and therefore $zy$ is an edge of $B$.
This proves~\ref{Ki}.
\end{proof}

Next we use the symmetry between $B$ and $R$ in order to obtain more 
information about maximum cliques in each of them.

\begin{theorem}
\label{K=L}
$K=L={{n-1} \over {2}} $ and $K,L \geq 3$.
\end{theorem}

\begin{proof}
From the symmetry between $B$ and $R$, we may assume that $K \geq L$. 
Since by~\ref{KL}, $\omega_B=K=n-1-L$, and by~\ref{n>5} $n \geq 6$,
it follows that $K+L =n-1 \geq 5$, and so $K \geq 3$. But now \ref{Ki}.2 
implies that $L \geq  K$. Thus $K=L={{n-1} \over {2}} $,  and \ref{K=L} follows.
\end{proof}

It now follows from~\ref{Ki}.3 and \ref{K=L} that there exists a vertex
$v_R \in V$ such that 
\begin{itemize}
\item $V \setminus \{v_R\} = Z \cup Y$, and
\item $Z \cap Y = \emptyset$, and
\item $Z$ is a clique of size $\frac{n-1}{2}$ in $R \setminus B$, and
\item $Y$ is a clique of size $\frac{n-1}{2}$ in $B$, and 
\item the vertices of $Z$ can be numbered $z_1, \ldots, z_K$, 
and the vertices of $Y$ can be numbered $y_1, \ldots, y_K$, such that
for $i,j \in \{1, \ldots, K\}$, the pair $z_iy_j \in B$ if and only if 
$i \neq j$.
\end{itemize}

Exchanging the roles of $R$ and $B$, we deduce also that there exists a vertex
$v_B \in V$ such that 
\begin{itemize}
\item $V \setminus \{v_B\} = Z' \cup Y'$, and
\item $Z' \cap Y' = \emptyset$, and
\item $Y'$ is a clique of size $\frac{n-1}{2}$ in $B \setminus R$, and
\item $Z'$ is a clique of size $\frac{n-1}{2}$ in $R$, and 
\item the vertices of $Z'$ can be numbered $z_1', \ldots, z_K'$,
and the vertices of $Y'$ can be numbered $y_1', \ldots, y_K'$, such that
for $i,j \in \{1, \ldots, K\}$, the pair $z_i'y_j' \in R$ if and only if 
$i \neq j$.
\end{itemize}

We now analyze the way $v_R$ attaches to $Y$ and $Z$.

\begin{theorem}
\label{Redge}
Let $i,j \in \{1, \ldots, K\}$. If $v_R$ is $B$-complete to $\{z_i,y_j\}$,
then $z_iy_j$ is an edge of $R \setminus B$.
\end{theorem}

\begin{proof}
Suppose that $v_R$ is $B$-complete to $\{z_i,y_j\}$ and $z_iy_j$ is an edge of 
$B$. Then $i \neq j$.
Since $(Y \cup \{v_R,z_i\}) \setminus \{y_i\}$ is not an clique of size $K+1$ in
$B$, it follows that there exists $t \in \{1, \ldots, K\}  \setminus \{i\}$
such that $v_Ry_t$ is an edge of $R \setminus B$. Then $t \neq j$. But now
$B|\{v_R,z_i,y_j,y_t,y_i,z_j\}$ is isomorphic to a member of $\mathcal{F}$, 
a contradiction. This proves~\ref{Redge}.
\end{proof}

We are finally ready to establish the existence of certain induced subgraphs
in $B$ and $R$.

\begin{theorem}
\label{subgraphs}
At least one of the following holds:
\begin{enumerate}
\item $B$ contains $P_0^c$, or
\item $B$ contains $P_1$ or $P_2$, and $v_R$ is $R \setminus B$-complete to 
$Y$, or
\item $B$ contains $P_0$, and there exists $z \in Z$ such 
that $v_R$ is $R \setminus B$-complete to $(Y \cup Z) \setminus \{z\}$.
\end{enumerate}
\end{theorem}

\begin{proof}
Since $Z \cup \{v_R\}$ is not a clique of size $K+1$ in $R$, it
follows that $v_R$ has a neighbor in $Z$ in $B \setminus R$. We may assume
that $v_Rz_1$ is an edge of $B \setminus R$. Since $z_1$ is $B$-complete
to $Y \setminus \{y_1\}$, \ref{Redge} implies that $v_R$ is 
$R \setminus B$-complete to $Y \setminus \{y_1\}$.

Suppose $v_R$ has a neighbor in $Z \setminus \{z_1\}$ in $B$, say
$v_Rz_2$ is an edge of $B$. Then by~\ref{Redge} $v_R$ is adjacent
in  $R \setminus B$ to $y_1$, and so $v_R$ is $R \setminus B$-complete to
$Y$. Also, $B|\{y_1,y_2,y_3,z_1,z_2,z_3,v_R\}$ is isomorphic to $P_1$
if $v_Rz_3$ is an edge of $R \setminus B$, and to $P_2$ if 
$v_Rz_3$ is an edge of $B$, and the second outcome of the theorem holds.

So we may assume that $v_R$ is $R \setminus B$-complete to 
$Z \setminus \{z_1\}$. Now if  $v_Ry_1$ is an edge of $B$, then
$B|\{y_1,y_2,y_3,z_1,z_2,z_3,v_R\}$ is isomorphic to $P_0^c$,
and the first outcome of the theorem holds;
and if $v_Ry_1$ is an edge of $R \setminus B$, then
$B|\{y_1,y_2,y_3,z_1,z_2,z_3,v_R\}$ is isomorphic to $P_0$,  $v_R$
is $R \setminus B$-complete to $(Y \cup Z) \setminus \{z_1\}$, and the
third outcome of the theorem holds. This proves~\ref{subgraphs}
\end{proof}

Applying \ref{subgraphs} with the roles of $R$ and $B$ reversed, we deduce that
either
\begin{enumerate}
\item $R$ contains $P_0^c$, or
\item $R$ contains $P_1$ or $P_2$, and $v_B$ is $B\setminus R$-complete to 
$Z'$, or
\item $R$ contains $P_0$, and there exists $y' \in Y'$, such that
$v_B$ is $B \setminus R$-complete to $(Y' \cup Z') \setminus \{y'\}$.
\end{enumerate}

To complete the proof of~\ref{main}, we now analyze the possible outcomes 
of~\ref{subgraphs}.  Observe first that by~\ref{subgraphs}, each of $B,R$ 
either contains $P_0^c$, or contains a member of $\mathcal{P}$.
Thus, if the first outcome of~\ref{subgraphs} holds for at least
one of $B,R$ (in other words, one of $B,R$ contains $P_0^c$), we get a 
contradiction to the third, fourth or fifth assumption at the start of 
Section~2.

So we may assume that either the second or the 
third  outcome of~\ref{subgraphs} holds for $B$, and the same for $R$.
Therefore $v_R$ is $R \setminus B$-complete to $Y$.  We claim that 
every vertex of $V$ has at least two neighbors in  $R \setminus B$.
Since by~\ref{K=L} $|Y|,|Z| \geq 3$, it follows that $v_R$ has at least two
neighbors in $Y$ in $R \setminus B$, and that every vertex of $Z$ has at least
two neighbors in $Z$ in $R \setminus B$. Since $v_R$ is $R \setminus B$-complete
to $Y$, and every vertex of $Y$ has a neighbor in $Z$ in $R \setminus B$,
the claim follows. Similarly, every vertex of $V$ has at least
two neighbors in $B \setminus R$.

Next we observe that if the third outcome of~\ref{subgraphs} holds for $B$,
then $v_R$ has at most one neighbor in $B$, and if the third outcome 
of~\ref{subgraphs} holds for $R$, then $v_B$ has at most one neighbor in $R$.
This implies that the third outcome of~\ref{subgraphs} does not hold for 
either $B$ or $R$, and thus the second outcome of~\ref{subgraphs} holds for 
both $B$ and $R$; consequently each of $B$ and $R$ contains $P_1$ or $P_2$. But
both $P_1$ and $P_2$ contain $C_5$, contrary to the second assumption at the
start of Section~2. This completes the proof of~\ref{main}. \bbox

\section{Acknowledgment}

We would like to thank Irena Penev for her careful reading of the manuscript, 
and for her helpful suggestions regarding its presentation.

\end{document}